\documentclass[11pt]{amsart}

\usepackage{amssymb, amsmath, amscd, dcpic, pictexwd}
\usepackage{hyperref}

\usepackage{amsopn}

\usepackage{tikz-cd}

\usepackage{tikz}
\usepackage{dsfont}
\usetikzlibrary{matrix}

\DeclareMathOperator{\HH}{H}

\DeclareMathOperator{\Ima}{Im}

\newtheorem{thm}{Theorem}[section]

\newtheorem{lem}[thm]{Lemma}
\newtheorem{prop}[thm]{Proposition}
\newtheorem{defn}[thm]{Definition}
\newtheorem{rem}[thm]{Remark} 
\newtheorem{exm}[thm]{Example}
\numberwithin{equation}{section}
\newcommand{\CC}{\mathbb C}
\newcommand{\PP}{\mathbb {P}}
\newcommand{\CA}{\mathbb{A}}
\newcommand{\QQ}{\mathbb{Q}}

\newcommand{\PPP}{\mathcal{P}}

\newcommand{\tP}{\tilde{P}}

\newcommand{\fg}{\mathfrak{g}}
\newcommand{\fp}{\mathfrak{p}}
\newcommand{\ft}{\mathfrak{t}}

\newcommand{\fb}{\mathfrak{b}}
\newcommand{\fsl}{\mathfrak{sl}}

\newcommand{\al}{\alpha}

\newcommand{\OX}{\mathcal{O}}

\newcommand{\tY}{\tilde{Y}}

\newcommand{\cN}{\mathcal N}

\newcommand{\cY}{\mathcal Y }

\renewcommand{\gg}{\mathfrak{g}}

\newcommand{\gs}{\mathfrak{s}}

\author{An Huang, Bong Lian, Shing-Tung Yau, Chenglong Yu}

\title{Jacobian rings for homogenous vector bundles and applications}
\begin{document}
\maketitle
\begin{abstract}
In this note, we examine the Jacobian ring description of the Hodge structure of zero loci of vector bundle sections on a class of ambient varieties. We consider a set of cohomological vanishing conditions that imply such a description, and we verify these conditions for some new cases. We also observe that the method can be directly extended to log homogeneous varieties. We apply the Jacobian ring to study the null varieties of period integrals and their derivatives, generalizing a result in \cite{CHL17} for projective spaces. As an additional application, we prove the Hodge conjecture for very generic hypersurfaces in certain generalized flag varieties.
\end{abstract}

\tableofcontents 
\baselineskip=16pt plus 1pt minus 1pt
\parskip=\baselineskip
\section{Introduction}
In \cite{Gr}, Griffiths introduced a Jacobian ring description for Hodge structures of hypersurfaces in a projective space. This was later generalized to hypersurfaces of sufficiently high degree by Green \cite{green}. Green's description involves the first prolongation bundle. For complete intersections or more general vector bundle sections, this was also studied by Flenner in connection to the local Torelli problem \cite{flenner} using spectral sequence argument and Koszul resolution. See \cite{terasoma} or \cite{konno} for details of Jacobian ring in terms of the first prolongation bundle for complete intersections or vector bundle sections. In \cite{Batyrevcox}, Batyrev and Cox obtained similar descriptions for toric hypersurfaces. It is generalized to complete intersections in toric varieties by Mavlyutov \cite{Mavlyutov}. The corresponding Jacobian ring is explicit in the sense that it only depends on the combinatorial data associated to a toric variety. 

In this note, we study the Jacobian rings for varieties given by zero loci of vector bundles sections from the point of tautological systems. The latter is general class of D-modules introduced earlier by the second and third authors \cite{LY}. We apply some of the arguments in earlier works to study the Hodge structures of this class of varieties. In the case of homogenous vector bundle on flag varieties, the Jacobian ring can be explicitly given in terms of representations of $G$. See Theorems \ref{main} and \ref{VJM}. Some of these descriptions may be well-known to experts. For example, in the case of hypersurfaces in irreducible Hermitian symmetric spaces, the description was given by Saito in \cite{saito}. We also show similar results for hypersurfaces in log homogenous varieties following Batyrev's work on toric hypersurfaces \cite{batyrev}. In section \ref{application}, we apply our results to study tautological systems -- differential systems associated to the period integrals of those families. In section \ref{Hodgeconjecture}, we prove Hodge conjecture for very generic hypersurfaces in certain generalized flag varieties.


We first fix some notations.
\begin{enumerate}
\item Let $G$ be a complex Lie group and $\fg=Lie(G)$
\item Let $X^n$ be a smooth projective variety together with action of $G$.
\item Let $E^r$ be a $G$-equivariant vector bundle on $X$ with rank $r$.
\item Assume $V^\vee=H^0(X, E)$ has basis $a_1, \cdots, a_N$ and dual basis $a_i^\vee$.
\item Let $f\in V^\vee$ be an section and $Y_f$ the zero locus of $f$. We further assume $Y_f$ is smooth with codimension $r$. See the discussion about smoothness in section \ref{Vector}.
\end{enumerate}

\subsection{Acknowledgment} The authors are grateful to Joe Harris and Colleen Robles for their interests and helpful discussions. 

\section{Line bundles}
\label{linebundles}
In this section, we consider the case that $E$ is an line bundle $L$. The Hodge structure of $Y_f$ is determined by ambient space $X$ by Lefschetz hyperplane theorem except the middle dimension. Let $U_f=X-Y_f$ and $H^{n-1}_{var}(Y)$ is the kernel of Gysin morphism $H^{n-1}(Y)\to H^{n+1}(X)$. Hodge structures of $Y$ and $U$ are related by the Gysin sequence
\begin{equation}
0\to H^n_{prim}(X)\to H^n(U)\to H^{n-1}_{var}(Y)(-1)\to 0
\end{equation}
Here $H^n_{prim}(X)$ is the cokernel of the Gysin morphism $H^{n-2}(Y)\to H^n(X)$. We will consider $X$ to be a Fano variety. So $H^n(X, \OX)=0$ and $F^0H^n(U)=F^1H^n(U)$.
\begin{defn}
\label{maindef}
Let $R$ be the graded ring $R=\oplus_{k\geq 0} H^0(X, L^k)$. The generalized Jacobian ideal $J$ is the graded ideal generated by $f, L_{Z} f$ for $Z\in \fg$. Here the Lie derivative $L_Z f$ is from the natural $\fg$-action on $H^0(X, L)$. Then $M=\oplus_{k\geq 0} H^0(X, K_X\otimes L^{k+1})$ is a graded $R$-module. 
\end{defn}
This definition gives Green's Jacobian ring \cite{green} under suitable vanishing conditions, see Proposition \ref{compare}. We have the following proposition.
\begin{prop}
\label{map}
There is an natural morphism from $$(M/JM)^k\to F^{n-k}H^{n}(U)/F^{n-k+1}H^{n}(U).$$It is compatible with multiplication map $H^0(X, L)\times (M/JM)^k\to (M/JM)^{k+1}$ and the Higgs field from Gauss-Manin connection
\end{prop}
This is a direct result from Griffiths' study of rational forms \cite{Gr}.
\begin{prop}
\label{localcal}
There is a natural map
\begin{equation}
\alpha_k\colon
{H^0(X, \Omega_X^n((k+1)Y))} \to F^{n-k}H^n(U)
\end{equation}
with $d H^0(X, \Omega_X^{n-1}(kY))$ contained in the kernel.
\end{prop}
\begin{proof}[Proof of Proposition \ref{map}]
See \cite{Voisin}, Theorem 6.5 for the construction of $\alpha_k$. Let $\bar{\alpha}_k$ be the induced map to the quotient $F^{n-k}H^{n}(U)/F^{n-k+1}H^{n}(U)$. Now we prove that $JM^{k-1}\subset \ker \bar{\alpha}_k$. The elements of $M^k$ can be realized as rational forms on $X$ as follows. Consider a G-equivariant principal bundle $\PPP$ over $X$ such that $L$ and $K_X$ are associated bundles. Then $f\in H^0(X, L)$ and $\tP\in H^0(X, K_X\otimes L^k)$ are viewed as functions on $\PPP$. There exists an $n$-form $\Omega$ defined on $\PPP$ such that $\Omega \tP\over f^{k+1}$ is the pull-back of a rational $n$-form on $X$. See \cite{LY} for the construction of $\Omega$ from principal bundle version of adjunction formula. We identify $\Omega^{n-1}$ with $T_X\otimes K_X$ and consider the morphism $\fg\otimes H^0(X, K_X\otimes L^{k})\to H^0(X,T_X\otimes L^{k}\otimes K_X)\cong H^0(X,\Omega_X^{n-1}(kL)) $ given by 
\begin{equation}
\sum_i Z_i\otimes T_i\mapsto \sum_i\iota_{Z_i} {{\Omega T_i}\over f^k}
\end{equation}
Here $Z_i$ is a basis of $\fg$, $T_i\in H^0(X, K_X\otimes L^{k})$ and $\iota$ is the contraction between vector fields with forms. Let $\gamma\in H^0(X, \Omega^{n-1}(kY))$ be such a differential form, then
\begin{equation}
\label{gamma}
d\gamma=d \sum_i\iota_{Z_i} {{\Omega T_i}\over f^k}=\sum_i L_{Z_i} {{\Omega T_i}\over f^k}
\end{equation}
by Cartan's formula. Assume $\tP$ satisfies 
 \begin{equation}
{\Omega  (\tP-fQ)\over f^{k+1}}=d\gamma
\end{equation}
for some $Q\in H^0(X, \Omega^n(kY))$. Then we have
\begin{align}
{\Omega  (\tP-fQ)\over f^{k+1}}&=\sum L_{Z_i} {\Omega T_i\over f^{k}}\\
&=\sum_i {{\Omega }}({(\beta(Z_i)T_i+L_{Z_i} T_i)f-kT_iL_{Z_i}f\over f^{k+1}})
\end{align}
Here the $G$-action lifts to $\PPP$ and the corresponding character on the $n$-form $\Omega$ is denoted by $\beta$. So $\tP=(Q+\sum_i\beta(Z_i)T_i+L_{Z_i} T_i)f-k\sum_iT_iL_{Z_i} f$. 
Since such $\tP\in \ker \bar{\alpha}_k$ and $Q$ can be any section in $H^0(X, K_X\otimes L^k)$, we have $JM^{k-1}\subset \ker \bar{\alpha}_k$. Hence this induces a map $(M/JM)^k\to F^{n-k}H^{n}(U)/F^{n-k+1}H^{n}(U)$. 

The Gauss-Manin connection on the universal family is given by differentiating $\Omega P\over f^{k+1}$.
\end{proof}

Now we discuss the case $X$ being homogenous and sufficient conditions for the above map being isomorphism. Let $X=G/P$ be a generalized flag variety. We consider the following two vanishing conditions for the line bundle $L$. 
\begin{equation}
\label{C}
H^p(X, \Omega_X^q\otimes L^l)=0 \text{ with $p>0, q\geq 0$, $l\geq 1$}  
\end{equation}
\begin{equation}
\label{C1}
H^1(X, (G\times_{ad{P}}\fp)\otimes L^{k}\otimes K_X)=0 \text{ for } k\geq 1
\end{equation}
We postpone the discussion about the conditions to next section.

\begin{thm}
\label{main}
Let $0\leq k\leq n-1$ and assume $L$ satiesfies conditions \eqref{C1} and \eqref{C} for $(p,q,l)$ in the following range $\{1\leq p\leq k, q=n-p, l=k-p+1\}\cup\{1\leq p\leq k-1, q=n-p-1, l=k-p\}\cup\{1\leq p\leq k-1, q=n-p, l=k-p\}$, then the map $$(M/JM)^k\to F^{n-k}H^{n}(U)/F^{n-k+1}H^{n}(U)$$ is an isomorphism.
\end{thm}

\begin{proof}
The proof follows the argument in Theorem 6.5 \cite{Voisin}. Under condition (\ref{C}) with $\{1\leq p\leq k, q=n-p, l=k-p+1\}\cup\{1\leq p\leq k-1, q=n-p-1, l=k-p\}$, the map
\begin{equation}
\alpha_k\colon
{H^0(X, \Omega_X^n((k+1)Y))} /d H^0(X, \Omega_X^{n-1}(kY))\to F^{n-k}H^n(U)
\end{equation}
is an isomorphism for $1\leq k\leq n-1$. Consider the exact sequence
\begin{equation}
0\to \fp\to \fg\to \fg/\fp \to 0.
\end{equation}
According to condition (\ref{C1}), the short exact sequence of vector bundles induces a surjective morphism $\fg\otimes H^0(X, K_X\otimes L^{k})\to H^0(X,\Omega_X^{n-1}(kL))$. So any $\gamma \in H^0(X,\Omega_X^{n-1}(kL))$ has the form in \eqref{gamma}. The map $\alpha_{k-1}\colon H^0(X, \Omega_X^n(kY)) \to F^{n-k+1}H^n(U)$ is surjective under condition \eqref{C} for $1\leq p\leq k-1, q=n-p, l=k-p$. So the kernel of $\bar{\alpha}_k$ is given by $JM^{k-1}$. 
\end{proof}

\begin{rem}
If $L$ and $K_X$ are multiples of the same ample line bundle $L^\prime$, then $R$ and $M$ can be embedded in the coordinate ring $R^\prime=\oplus H^k(X, (L^\prime)^k)$. We can define the Jacobian ideal $J^\prime$ to be the ideal generated by $f, L_Z f$ in $R^\prime$. Then the degree-$k$ summand $(M/JM)^k$ is the corresponding summand in $R^\prime/J$. When $X=\PP^n$, we can take $L^\prime$ to be $\OX(1)$ bundle. Let $L=\OX(d)$. The vanishing conditions are satisfied except $kd=n$ for condition \ref{C1}. When $kd\neq n$, this is the same Jabocian ring description for Hodge structures of hypersurfaces in projective spaces. More specifically, the Jacobian ideal defined here are generated by $x_j{\partial f\over \partial x_i}$ if we view $f$ as polynomial of homogenous coordinates $[x_0,\cdots, x_n]$. The usual Jacobian ideal are generated by ${\partial f\over \partial x_i}$. When $kd\neq n$, the corresponding degree part in the usual Jacobian ring are quotients of elements in the form $\sum_i g_i {\partial f\over \partial x_i}$ with $g_i$ homogenous with degree greater than 0, which are the same degree part of ideal generated by $x_j{\partial f\over \partial x_i}$.
\end{rem}

\begin{rem}
For hypersurfaces in irreducible Hermitian symmetric spaces, the Jacobian ring defined here is already given by Saito in \cite{saito}. See Lemma 4.1.12 \cite{saito}. The vanishing conditions required here is slightly weaker than the ones used in \cite{saito}. This is because the approach in Theorem 6.5 \cite{Voisin} used the exactness of log de-Rham complex in degree $\geq 2$ with smooth $Y$. So the cohomology of closed log forms $H^{k}(X, \Omega_X^{n-k,c}(\log Y))$ computes the hypercohomology of $$0\to \Omega_X^{n-k}(\log Y)\to \Omega_X^{n-k+1}(\log Y)\to \cdots.$$ On the other hand, $\Omega_X^{n-k,c}(\log Y)\cong \Omega_X^{n-k,c}(Y)$ has a resolution 
$$ 0\to \Omega_X^{n-k}(Y)\to  \Omega_X^{n-k+1}(2Y)\to \cdots .$$
The proof in \cite{saito} used the exact sequence
\begin{eqnarray*}
0\to \Omega_{X}^{n-k}(\log Y)\to \Omega_{X}^{n-k}(Y)\xrightarrow{d} \Omega_{X}^{n-k+1}(2Y)/\Omega_{X}^{n-k+1}(Y)\xrightarrow{d}  \\ \cdots \xrightarrow{d}  \Omega_{X}^{n}((k+1)Y)/\Omega_{X}^{n}(kY)\to 0
\end{eqnarray*}
The vanishing $$H^{k-l}(X, \Omega_{X}^{n-k+l}((l+1)Y)/\Omega_{X}^{n-k+l}(lY))=0$$ and $$H^{k-l-1}(X, \Omega_{X}^{n-k+l}((l+1)Y)/\Omega_{X}^{n-k+l}(lY))=0$$ requires more vanishing conditions \eqref{C}. See Section \ref{log} for the application of this exact sequence in the log homogenous case, which follows the same argument for hypersurfaces in algebraic tori \cite{batyrev}.
\end{rem}

\begin{rem}
The Kodaira-Spencer map $H^0(X, L)\to H^1(Y, T_Y)$ has kernel equal to $J$ if $H^1(X, T_X\otimes L^{-1})=0$, and is surjective if $H^2(X, T_X\otimes L^{-1})=0$. In this case, the multiplication map $H^0(X, L)/J\times (M/JM)^k\to (M/JM)^{k+1}$ gives the Higgs field in universal deformation family of $Y$. For example this holds for Grassmannians $G(a, b)$ with $b\geq 5$ with any ample line bundle $L$.
\end{rem}

\begin{rem}
In order to get similar description as $\PP^n$ and coordinate ring $\CC[a_1,\cdots a_{n+1}]$, we consider $M$ to be the total space of $H$-principal bundle over $X$ with $G$-equivariant action. The character $\chi\colon H\to \CC^*$ is associated with the line bundle $L$. In many cases, the total space $M$ is embedded in affine space $\bar{M}$ as Zariski open set and global sections of structure sheaf of $M$ is extended to $\bar{M}$. Assume that the $G$-action also extends to $\bar{M}$. For example, when $X=G(a,b)$ is the Grassmannian, we can take $M$ to be the Stiefel manifold and $\bar{M}$ the affine space $\CA^{ab}$. The coordinate ring $R$ is identified with $\CC[\bar{M}]^{H,\chi}$, which is the functions that are equivariant under the $H$ action by characters $m\chi$. The basis can be given by standard monomials. Then the sections $f$ and $L_Z f$are identified with elements in $\CC[\bar{M}]$ similar as the $\PP^n$ case.
\end{rem}

Now we discuss the relation to Green's Jacobian ring \cite{green}. See Saito's identification of two definitions for Hermitian symmetric spaces, Lemma 3.2.3 \cite{saito}. First we recall Green's definition of Jacobian ring. Let $\Sigma_L$ be the first prolongation bundle. It is the bundle of first order differential operators on $L$. The differentiation of $f$ gives a section $df\in H^0(X, \Sigma_L^*\otimes L)$. This induces a map $H^0(X, L^k\otimes K_X\otimes \Sigma_L)\to H^0(X, L^{k+1}\otimes K_X)$. The Jacobian ring $R_k$ is the cokernel of this map.
\begin{prop}
\label{compare}
If $L$ is ample and satisfies condition \ref{C1}, then $(M/JM)^k\cong R_k$.
\end{prop}

\begin{proof}
The proof is the same as Lemma 3.2.3 \cite{saito}. Consider the exact sequence 
$$
0\to O_X\to \Sigma_L\to T_X\to 0
$$
twisted by $L^k\otimes K_X$. We have $$0\to H^0(X, L^k\otimes K_X)\to H^0(X, L^k\otimes K_X\otimes \Sigma_L)\to H^0(X, L^k\otimes K_X\otimes T_X)\to 0.$$
Since $\fg\otimes H^0(X, L^k\otimes K_X)\to H^0(X, L^k\otimes K_X\otimes T_X)$ is surjective, similar calculation as Proposition \ref{localcal} shows that the image of $H^0(X, L^k\otimes K_X\otimes \Sigma_L)\to H^0(X, L^{k+1}\otimes K_X)$ is $JM^{k-1}$.
\end{proof}

\section{Vanishing conditions}
In this section we discuss the vanishing conditions, especially when $Y$ is Calabi-Yau. First we discuss condition \eqref{C1} when $L=-K_X$.

For $l=1$, the vanishing condition \eqref{C1} is equivalent to $H^0(X, T_X)=\fg$. When $G$ is simple, there are three series of exceptional cases that this condition fails. See \cite{Laction}, Chapter 3.3, Theorem 2.

For $l\geq 2$, we prove that vanishing condition \eqref{C1} holds for any $X=G/P$ with $L=-K_X$.
\begin{prop}
\label{c1}
If $L=-K_X$, then we have $H^i(X, (G\times_{ad{P}}\fp)\otimes L^{k})=0 \text{ for } i\geq 1, k\geq 1$
\end{prop}
\begin{proof}
Without loss of generality, we assume $G$ is simple. We first fix some notations. Let $\Phi_+$ be the set of positive roots for $\fg$ and $B_-\subset P$ be the corresponding negative Borel subgroup. Let $\fp=\ft\oplus \bigoplus_{\alpha\in \Phi_-\cup \Phi_P} {\fg_\alpha}$, where $\Phi_P$ is the set of positive roots generated by $\{\alpha_1,\cdots, \alpha_l\}\backslash \{\alpha_{j_1}, \cdots \alpha_{j_m}\}$. The pairing on $\ft^\vee$ induced by the Killing form is denoted by $(,)$. The pairing $\langle, \rangle$ is defined by $\langle \alpha, \beta\rangle=2{(\alpha, \beta)\over (\beta,\beta)}$. Consider the Jordan-H\"older $\fp$-representation filtration of $\fp=V_0\supset V_1\supset\cdots \supset V_t$ with irreducible factors $W_i=V_i/V_{i+1}$. The corresponding associated bundles over $X$ are also denoted by $V_i$ and $W_i$.
Then the highest weight of $W_i$ as maximal semisimple Lie subalgebra $\fp^{ss}\subset \fp$, denoted by $\lambda_i$, is either $0$ or in $\Phi_-\cup \Phi_P$. The weight for $L$ is $2\rho_P=\sum_{\alpha\in \Phi_+\backslash\Phi_P}\alpha$. 

Now we prove all the higher cohomology groups of $W_i$ vanish. The complement of $\{\alpha_{j_1}, \cdots ,\alpha_{j_m}\}$ in the Dynkin diagram is decomposed as connected components $D_1, D_2, \cdots, D_a$. The semisimple part $\fp^{ss}$ is the direct sum of Lie subalgebras corresponding to $D_j$. Then $\lambda_i$ restricted to the Cartan of $\fp^{ss}$ is dominant weight for $\fp^{ss}$. So $\langle \lambda_i,\beta\rangle \geq 0$ for all $\beta\in \Phi_P$. The paring used here is induced by the Killing form on $\fp^{ss}$. Since Killing form is the unique bilinear pairing invariant under adjoint action up to rescaling, we use the same notation. We claim $\langle \lambda_i+2k\rho_P+\rho,\beta\rangle \geq 0$ for any $\beta\in \Phi_+$. For simple root $\beta$, we have $\langle \rho_P, \beta\rangle\geq 0$ and $\langle \rho, \beta\rangle=1$. If $\beta\in \{\alpha_1,\cdots, \alpha_l\}\backslash \{\alpha_{j_1}, \cdots \alpha_{j_m}\}$, then $\langle \lambda_i, \beta\rangle\geq 0$. If $\beta=\alpha_{j_b}$, then $\langle \lambda_i, \beta\rangle\geq -3$. In this case, we have $\langle \rho_P, \alpha_{j_b}\rangle\geq 1$ since $\langle 2\rho-2\rho_P, \alpha_{j_b}\rangle \leq 0$. Here $2\rho-2\rho_P$ is the sum of positive roots in $D_j$, hence has nonpositive product with $\alpha_{j_b}$.
\end{proof}

The above method also proves the vanishing condition \eqref{C1} for irreducible Hermitian symmetric spaces.
\begin{prop}
\label{c11}
Then condition \ref{C1} holds for any ample line bundle on irreducible Hermitian symmetric space $X$ that is not isomorphic to projective spaces. Here $G$ is $Aut(X)$.
\end{prop}
\begin{proof}
Let $X=G(a,l+1)$ be a Grassmannian not isomorphic to $\PP^l$. Let $\OX(1)$ be the positive generator of Picard group of $X$. Then $L^k\otimes K_X=\OX(t)$ for some integer $t$. The tangent bundle $TX$ is associated bundle of $\fg/\fp$ and $\fg/\fp$ is irreducible $\fp^{ss}$ representation with highest weight being the highest long root. We use the previous notations and $\Phi_P$ is generated by $\{\alpha_1,\cdots, \alpha_l\}\backslash \{\alpha_{j}\}$. Let $\omega_i$ be the $i$th fundamental weight. Then the highest long root is $\omega_1+\omega_l$. The line bundle $\OX(1)$ is the associated line bundle with weight $\omega_j$ with $j\neq 1, l$. If $t\geq 0$, then $H^0(X, L^k\otimes K_X)\neq 0$ and $H^0(X, T_X\otimes L^k\otimes K_X)$ is an irreducible $\fg$-module. So the map $\fg\otimes H^0(X, L^k\otimes K_X)\to H^0(X, T_X\otimes L^k\otimes K_X)$ is surjective. If $t<0$, then $H^0(X, T_X\otimes L^k\otimes K_X)=0$ since $\omega_1+\omega_l-t\omega_j$ is not dominant weight. Let $X$ be any other irreducible Hermitian symmetric space $G/P$. Let $P$ be maximal parabolic subgroup with $\alpha_j$ removed. The highest long root $\lambda=\sum a_i \omega_i$. If $X$ is not isomorphic to projective spaces, then $a_j=0$. The argument for Grassmannian applies to this case. 
\end{proof}


Now we discuss vanishing condition \eqref{C}. When $X$ is Hermitian symmetric, the homogeneous bundle is decomposed as direct sum of irreducible vector bundles and the corresponding highest weights are given by the criterion of Kostant \cite{Kostant}. In \cite{Snow1}, Snow has found some sufficient conditions. See the Proposition in Section 1 of \cite{Snow1}. Especially for Grassmannian $X=G(a, l+1)$, vanishing condition \ref{C} holds for $L=\OX(t)$ with $t\geq l$. When $t<l$, there are some cases not satisfying the vanishing condition. For instance, the cohomology $H^p(X, \Omega^{n-p}\otimes \OX(2))\neq 0$ when $X=G(a, 2a)$ and $p={a^2-a\over 2}$. See \cite{SnowG} Theorem 3.2 and 3.3.

\begin{exm}
Here we show some examples for which vanishing condition \eqref{C} holds. The upshot is that the homogenous bundles involved are not direct sum of irreducible vector bundles and the irreducible factors in Jordan-H\"older filtration admits nontrivial higher cohomology.

Let $X^n=G/B$ with $\fg=\fsl_{l+1}$ and $L=K_X^{-1}$ with $l=1,2,3,4$. Then we have $\Omega^p\otimes L\cong \wedge^{n-p} T_X$. When $l=1,2$, we have $H^i(X, \wedge^j T_X)=0, i\geq 0$ from Borel-Weil-Bott. For $l=3$, the only remaining case is $H^i(X, \wedge^2 T_X)=0$ for $i\geq 1$. Since $T_X\cong G\times_B (\fg/\fb)$, it has a natural filtration given by the heights of positive roots $V_0=\fg^3/\fb_-\supset V_1=\fg^2/\fb_-\supset V_2=\fg^1/\fb_- \supset 0$ with factors 
\begin{equation}
W_i=V_i/V_{i+1}\cong \bigoplus_{ht(\alpha)= 3-i}L_\alpha
\end{equation}
The induces a filtration $\wedge^2 V_0=V_0\wedge V_1\supset \cdots \wedge^2 V_2$ with successive quotients $W_{01} =W_0\otimes W_1, W_{02}=W_0\otimes W_2, W_{11}=\wedge^2 W_1, W_{12}=W_1\otimes W_2, W_{22}=\wedge^2 W_2$.

Then we have the following spectral sequence from the filtration
\begin{equation}
\begin{tikzpicture}
  \matrix (m) [matrix of math nodes,
    nodes in empty cells,nodes={minimum width=5ex,
    minimum height=5ex,outer sep=-5pt},
    column sep=1ex,row sep=1ex]{
                &      &     &     & &&\\
          0\quad     &   210\oplus 012 &  0  & 0 & 0&0&\\
          -1\quad    &  0  & 103\oplus 301\oplus 020 &  0  &0&0& \\
          -2 \quad   &  0  & 0 &  020 & 020\oplus 020& 0 &\\
          -3 \quad   &  0 & 0 &  0  & 101\oplus 101&101&\\
          -4  \quad  &   0 & 0 &  0  &0 &0&\\        
    \quad\strut &   0  &  1  &  2  & 3& 4& \strut \\};
\draw[thick] (m-1-1.east) -- (m-7-1.east) ;
\draw[thick] (m-7-1.north) -- (m-7-7.north) ;
\end{tikzpicture}
\end{equation}
The terms $E_1^{p,q}$ in the spectral sequence are $H^{p+q}(W_p)$ with $W_{ij}$ reindexed and $m_1m_2m_3$ stands for irreducible $\fg$-module with highest weight $\sum_i m_i \omega_i$, where $\omega_i$ is the $i$-th fundamental weight. Now we study the differential $101\oplus 101\to 101$. The $E_1^{4,-3}$ component is $H^1(W_4)$ with $W_4$ generated by $v_1=e_{\alpha_1}\wedge e_{\alpha_3}$ and $E_1^{3,-3}$ is $H^0(W_3)$ with $W_3$ generated by $v_2=e_{\alpha_1+\alpha_2}\wedge e_{\alpha_3}, v_3=e_{\alpha_2+\alpha_3}\wedge e_{\alpha_1}$. Notice that $\sigma_{\alpha_2}(\alpha_2+\alpha_3+\rho)-\rho=\alpha_1+\alpha_2+\alpha_3=101$. Denote the minimal parabolic subgroup generated by $\alpha_2$ by $P_{\alpha_2}$. Consider the projection map $\pi_{\alpha_2}\colon G/B\to G/P_{\alpha_2}$ and the spectral sequence converging to $R^*{\pi_{\alpha_2}} V_0$ induced by the same filtration. The fibers for the bundles in the spectral sequence are determined by the action of the Borel of $\fsl_2$ generated by $\alpha_2$. The short exact sequence is given by 
\begin{equation}
\label{es}
0\to W_4\to V_3\to W_3\to 0
\end{equation}
with action of $\fsl_2$ on 
\begin{align}
f_{\al_2}{v_1}=0, f_{\al_2}{v_2}=v_1, f_{\al_2}{v_3}=-v_1\\
h_{\al_2}{v_1}=-2v_1, h_{\al_2}{v_2}=0, h_{\al_2}{v_3}=0
\end{align}
If we consider the subspace $\tilde{V}_3$ generated by $v_1, v_2(\text{or }v_3)$, then the restriction of $\ref{es}$ on $\PP^1=P_{\al_2}/B$ is isomorphic to
\begin{equation}
0\to \OX(-1)\to \CC^2\to \OX(1)\to 0
\end{equation}
twisted by $\OX(-1)$. 
Hence we have $R^1\pi_{\al_2}(\tilde{V}^3)|_{\PP^1}\cong H^0(\PP^1,\OX(-1))\otimes \CC^2\cong 0$ and $R^0\pi_{\al_2}(\tilde{V}^3) \cong 0$. The boundary map $R^0\pi_{\al_2}(W_3)\to R^1\pi_{\al_2}(W_4)$ is isomorphism on each component. From the following commutative diagram
\begin{equation}
\begin{tikzcd}
H^0(G/P_{\al_2},R^0\pi_{\al_2}(W_3))\arrow{r}\arrow{d} & H^0(G/P_{\al_2}, R^1\pi_{\al_2}(W_4))\arrow{d} \\
H^0(G/B, W_3) \arrow{r} & H^1(G/B, W_4)
\end{tikzcd}
\end{equation}
We conclude that $d_1\colon E_1^{3,-3}\to E_1^{4,-3}$ is isomorphism restricted to each component. 

The same argument shows that $d_1\colon E_1^{2,-2}\to E_1^{3,-2}$ is isomorphism when projected to each components of $E_1^{3,-2}$. But the second differential is difficult to calculate this way.

There is another approach for calculating the cohomology for homogenous bundles. See proposition 2.8 in \cite{Qi}. It combines the BGG resolution of Verma modules and Bott's theorem on cohomology of homogenous vector bundles. The highest weight $\lambda$ part of the cohomology of vector bundle $G\times_B E$ is given by the following sequence
\begin{equation}
0\to E[\lambda]\to \oplus_{w\in W, l(w)=1}E[w\lambda]\to \cdots\to E[w_0\lambda]\to 0
\end{equation}
Apply this sequence to our situation, we only need to check $\lambda=(020)$ and $E=\wedge^2(\fg/\fb)$. The first differential is realized as multiplication by $\pm f_i^{(\lambda, \alpha_i^{\vee})+1}$. In this case, we have $d_1^*(e_{\alpha_1+\alpha_2+\alpha_3}\wedge e_{\alpha_2})=f_1(e_{\alpha_1+\alpha_2+\alpha_3}\wedge e_{\alpha_2})\pm f_3(e_{\alpha_1+\alpha_2+\alpha_3}\wedge e_{\alpha_2})=e_{\alpha_2+\alpha_3}\wedge e_{\alpha_2}\pm e_{\alpha_1+\alpha_2}\wedge e_{\alpha_2}$ and $d_1^*(e_{\alpha_1+\alpha_2}\wedge e_{\alpha_2+\alpha_3})=f_1(e_{\alpha_1+\alpha_2}\wedge e_{\alpha_2+\alpha_3})\pm f_3(e_{\alpha_1+\alpha_2}\wedge e_{\alpha_2+\alpha_3})=e_{\alpha_2}\wedge e_{\alpha_2+\alpha_3} \pm e_{\alpha_1+\alpha_2}\wedge e_{\alpha_2}$. The choices of plus or minus sign are the same in two expressions. Hence $d_1^*$ is surjective. We have $H^1(X, \wedge^2 TX)=0$. When $l=4$, the same calculation shows vanishing $H^i(X, \wedge^j T_X)=0, i> 0$.

\end{exm}

\section{An application to differential zeros of period integrals}
\label{application}
Let $X$ be an $n$-dimensional smooth projective variety such that its anti-canonical line bundle $L:=K_X^{-1}$ is very ample.  Let $G$ be a connected algebraic group acting on $X$. We assume that $G$ acts with finitely many orbits. We shall regard the basis elements $a_i$ of $V=\Gamma(X,L)^{\vee}$ as linear coordinates on $V^{\vee}$. Let $B:=\Gamma(X,L)_{sm}\subset V^{\vee}$ be the space of smooth sections.

 Let $\pi:\cY\rightarrow B
$
 be the family of smooth CY hyperplane sections $Y_a\subset X$, and let $\HH^{\text{top}}$ be the Hodge bundle over $B$ whose fiber at $a\in B$ is the line $\Gamma(Y_a,\omega_{Y_a})\subset H^{n-1}(Y_a)$. In \cite{LY} the period integrals of this family are constructed by giving a canonical trivialization of $\HH^{\text{top}}$. Let $\Pi$ be the period sheaf of this family, i.e. the locally constant sheaf generated by the period integrals.
 
\cite{CHL17} initiated a study of the zero loci of derivatives of period integrals under this canonical trivilization. Namely, For any $\delta\in D_{V^\vee}$, denote
\begin{equation}\label{Np}
\cN(\delta)=\{b\in B\mid \delta \gs(b)=0,\, \forall \text{periods }\gs\}.
\end{equation}

In \cite{CHL17}, it has been shown that $\cN(\delta)$ is algebraic, and in the case when $X=\PP^n$, an explicit equation for $\cN(\delta)$ was given -- see Thm 7.2 in \cite{CHL17}. This explicit equation in particular gives a natural stratification of the zero loci. In \cite{CHL17}, based on the theory of tautological systems applied to general homogeneous varieties $G/P$, Thm 7.2 is a direct consequence of Lemma 7.1, whose proof relies on the algebraic and geometric rank formulae (Thm 2.9 in \cite{BHLSY}, and Thm 1.4 in \cite{HLZ}) for tautological systems that both apply to general $G/P$, together with the Jacobian ring description of the Hodge structure of hypersurfaces in $\PP^n$. Now, the generalization of the latter to certain cases of homogeneous varieties directly provides a generalization of Thm 7.2 to those cases, where the statement stays the same word by word. In the following, we state and prove the generalization of Lemma 7.1, therefore establishing the generalization of Thm 7.2 as mentioned. We will follow notations in \cite{HLZ}\cite{BHLSY}. In particular, $H_0(\hat{\gg},Re^{f(b)})$ denotes a particular Lie algebra coinvariant space, which appears as the right hand side of the algebraic rank formula for tautological systems -- see section 2 of \cite{BHLSY}.

\begin{lem}[Degree bound lemma]\label{dbl}
Let $X$ be a homogeneous variety $G/P$ of dimension $n$, where Theorem \ref{main} holds,\, let $\hat{\gg}=Lie(G)\oplus\CC.$ Let $Z_i$ be a basis of $\hat{\gg}$. Suppose $f(b)$ is nonsingular. For $h\in R$,  $he^{f(b)}\equiv0$ in $H_0(\hat{\gg},Re^{f(b)})$ iff $$he^{f(b)}=\sum Z_i(r_i e^{f(b)})$$ for some $r_i\in R,$ and $\deg r_i\leq\deg h-1,\,\forall i.$
\end{lem}
\begin{proof}

The `if' direction is obvious. For the `only if' direction, recall the homogeneous Jacobian ideal $J:=< Z_i {f(b)}| Z_i\in\hat{\fg}>$ of $R$. Let $B_k$ denote a $\CC$-basis for the degree $k$ part of $R/J$. First, since ${f(b)}$ is homogeneous of degree $1$, the degree $0$ part of $R/J$ is nonzero, and is spanned by $1$. For any $h\in R$, consider expanding the highest degree component of $h$, which we denote by $h_0$, in degree = $\deg h$ part of $R/J$ in terms of the chosen basis: i.e. by definition, there exist elements $s_i\in R$, such that $h_0-\sum s_iZ_i(f(b))$ can be written as a linear combination of the chosen basis elements in degree = $\deg h$. Obviously, we can require that $\deg s_i\leq \deg h-1$ for each $i$ by dropping all higher degree components of each of these $r_i$, if there are any. Working degree by degree, it is clear that we can choose $r_i\in R$ with $\deg r_i\leq \deg h-1, \forall i$, such that $he^{f(b)}=\sum Z_i(r_i e^{f(b)})+\sum c_kB_k$, where $\sum c_kB_k$ denote a linear combination of elements of the $B_k$ with all $k\leq \deg h$. Therefore, $H_0(\hat{\gg},Re^{f(b)})$ is spanned by $B_k$.

On the other hand, by Theorem \eqref{main} applied to the Calabi-Yau case $L=K_X^{-1}$, and taking the sum over $k$, one has $\dim_{\CC} R/J=H^{n}(X-V({f(b)}))$. Combining the algebraic and geometric rank formula for tautological systems applied to $G/P$\cite{BHLSY}\cite{HLZ}, we have in this case, $h^{n}(X-V({f(b)}))=\dim H_0(\hat{\gg},Re^{f(b)})$. Therefore, the collection of $B_k$ consists of linearly independent elements, and $he^{f(b)}=0$ in $H_0(\hat{\gg},Re^{f(b)})$ iff all coefficients $c_k=0$.
\end{proof}

\begin{rem}
Here our grading convention on $R$ is such that $f(b)$ is of degree 1. $1\in\CC$ acts on $V^{\vee}$ by identity. In the algebraic rank formula for tautological systems, there is a twist of the action of the Euler operator by a constant, however this twist does not affect the above proof, since the degree 0 part of $R/J$ is of dimension 1, and is spanned by 1.
\end{rem}


\section{Zero loci of vector bundle sections}
\label{Vector}
In this section we describe Hodge structure of zero loci of vector bundle sections on homogenous variety $X=G/P$. We first fix some notations.
\begin{enumerate}
\item Let $G$ be a semisimple algebraic group with parabolic subgroup $P$ and $X=G/P$ be the corresponding flag variety.
\item Let $\rho\colon P\to GL(r)$ be a semisimple representation of $P$. The associated homogenous vector bundle is $E=G\times_P \CC^r$. 
\item Consider the space of global sections $V^\vee=H^0(X, E)$ and $f\in V^\vee$. The zero locus is defined by $Y_f=\{f=0\}$.
\end{enumerate}

There are two different typical examples, one is that $E$ is the direct sum of line bundles $E=L_1\oplus L_2\cdots \oplus L_r$, the other is that $\rho$ is an irreducible representation of $P$. The idea follows the Cayley trick in the description of complete intersections. See Konno's paper \cite{konno} for Green's Jacobian ring and application to Torelli theorem for complete intersections. For complete intersections in toric varieties, see \cite{Mavlyutov}. Let $\PP=P(E^\vee)$ be the projectivation of $E^\vee$ and $\OX(1)$ be the hyperplane section bundle on $\PP$. The projection map is denoted by $\pi\colon \PP\to X$. From now on, we assume $E$ is ample and by definition, is equivalent to $\OX(1)$ being ample. We collect the propositions relating the geometry of $X$ and $\PP$ in the following. Proposition \ref{prop1} and \ref{prop2} are from \cite{terasoma},\cite{konno} and \cite{Mavlyutov}. Proposition \ref{PCV} is from Corollary 4.9 in \cite{konno}.

\begin{prop}
\label{prop1}
\begin{enumerate}
\item There is a natural isomorphism $H^0(X, E)\cong H^0(\PP, \OX(1))$. The corresponding section in $H^0(\PP, \OX(1))$ is also denoted by $f$. 
\item Let $\tY$ be the zero locus of $f$ in $\PP$. Then $\tY$ is smooth if and only if $Y$ is smooth with codimension $r$ or empty. 
\item There is an natural isomorphism $K_\PP\cong \pi^*(K_X\otimes \det E)\otimes \OX(-r)$
\end{enumerate}
\end{prop}

From now on, we assume $Y$ is smooth with codimesion $r\geq 2$. 
\begin{defn}
The variable cohomology $H^{n-r}_{var}(Y)$ is defined to be cokernel of $H^{n-r}(X)\to H^{n-r}(Y)$. 
\end{defn}

\begin{prop}
\label{prop2}
There is an isomorphism 
\begin{equation}
H^{n+r-1}(\PP-\tY)\cong H^{n+r-2}_{var}(\tY)(-1)\cong H^{n-r}_{var}(Y)(-r)
\end{equation}
\end{prop}


In order to describe the Hodge structure on $H_{var}^{n-r}(Y)$, we consider the following vanishing conditions. 
\begin{equation}
\label{CV2}
\begin{array}{lr}
&H^p(X, \Omega_X^{q-a} \otimes \wedge^{a}E{\otimes S^{l-a}E})=0 \text{ for $p>0, q\geq 0$ and $0\leq a\leq l-1$}\\
\text{or }&H^p(X, \Omega_X^{q-a} \otimes \wedge^{a+1}E{\otimes S^{l-a-1}E})=0 \text{ for $p>0, q\geq 0$ and $0\leq a\leq l-1$}
\end{array}
\end{equation}
\begin{equation}
\label{C1V}
H^1(X, (G\times_{ad{P}}\fp)\otimes S^{k}(E)\otimes \det E\otimes K_X)=0 \text{ for } k\geq r
\end{equation}

\begin{prop}
\label{PCV}
If $E$ satisfies the vanishing condition \ref{CV2}, then $\PP$ satisfies the vanishing condition
\begin{equation}
\label{CV}
H^p(\PP, \Omega_\PP^q\otimes \OX(l))=0 \text{ for $p>0, q\geq 0$ and  $l\geq 1$}  
\end{equation}
\end{prop}
\begin{proof}
Consider the exact sequence 
\begin{equation}
0\to \ker \pi_*\to T_\PP\to \pi^*T_X\to 0
\end{equation}
We denote $\ker \pi_*$ by $T_v$ and $(\ker \pi_*)^\vee$ by $\Omega_v$ for simplicity. The bundle $\wedge^q\Omega_\PP$ admits a filtration  with graded pieces $Gr^a\cong \wedge^a \Omega_v\otimes \wedge^{q-a}\pi^*\Omega_X$. It is sufficient to prove $H^p(\PP, Gr^a\otimes \OX(l))=0$. Consider the Leray spectral sequence $$E_2^{b,c}=H^b(X, R^c\pi_*(Gr^a))\to H^{b+c}(\PP, Gr^a).$$
According to projection formula, we have 
\begin{equation}
R^c\pi_*(Gr^a)\cong R^c\pi_*(\Omega_v^{a}\otimes \OX(l))\otimes \wedge^{q-a}\Omega_X
\end{equation}
The fiber of $R^c\pi_*(\wedge^a\Omega_v\otimes \OX(l))$ over a point $x$ is canonically isomorphic to $H^c(\PP(E_x^\vee), \Omega_{\PP(E_x^\vee)}^a\otimes \OX(l))$. If $c>0$ or $l\leq a$, this is zero according to Bott vanishing theorem. 
If $c=0, l\geq a+1$, then $H^0(\PP(E_x^\vee), \Omega_{\PP(E_x^\vee)}^a\otimes \OX(l))$ is an irreducible representation of $GL(E_x)$. Claim this is an irreducible factor of $\wedge^a(E_x)\otimes S^{l-a}(E_x)$ as $GL(E_x)$-representation. The proof of the claim is in the following lemma. 
\end{proof}

\begin{lem}[Bott]
\label{fiber}
Let $W$ be $r$-dimensional vector space. Let $\PP(W^\vee)$ be projective space consisting of the lines in the dual vector space. Assume $1\leq a\leq l-1$ and $\Omega^a\otimes \OX(l)$ has a natrual $GL(W)$-action induced by the tautological action on $\OX(1)$. Then $H^0(\PP(W^\vee), \Omega^a\otimes \OX(l))$ is an irreducible factor of $\wedge^a(W)\otimes S^{l-a}(W)$ or $\wedge^{a+1}(W)\otimes S^{l-a-1}(W)$ as $GL(W)$-representation.  
\end{lem}
\begin{proof}
The Euler sequence
\begin{equation}
0\to \Omega^1\to W\otimes \OX(-1)\to \mathbb{C}\to 0
\end{equation}
is an exact sequence of homogenous $GL(W)$-bundle. It induces an exact sequence
\begin{equation}
0\to \Omega^a\to \wedge^a W\otimes \OX(-a)\to \Omega^{a-1}\to 0
\end{equation}
There is an exact sequence of $GL(W)$-representations
\begin{equation}
0\to H^0(\PP(W^\vee), \Omega^a\otimes \OX(l))\to \wedge^a(W)\otimes S^{l-a}(W)\to H^0(\PP(W^\vee), \Omega^{a-1}\otimes \OX(l))\to0
\end{equation} 
Hence it is an irreducible summand of $\wedge^a(W)\otimes S^{l-a}(W)$.

\end{proof}


\begin{defn}
\label{MJN}
Let $S^k(E)$ be the symmetric product of $E$. Then the coordinate ring of $\PP$ is $R=\oplus_{k\geq 0} H^0(X, S^k(E))$ graded by $k$. Let $M=\oplus_{k\geq r-1} H^0(X, S^{k+1-r}(E)\otimes \det E \otimes K_X)$ be a graded $R$-module with gradings $k$. The Jacobian ideal $J$ is the ideal in $R$ generated by $f$ and $L_Zf, Z\in \fg$. Denote $N_k=H^0(X, E^\vee \otimes S^{k+1-r}(E)\otimes  \det E \otimes K_X)$. There is a map from $N_k$ to $M_k$ defined by pairing the $E^\vee$ component with $f$.
\end{defn}

\begin{thm}
\label{VJM}
If $E$ satisfies the vanishing conditions \eqref{C1V}, \eqref{CV2} for $p, q, l$ in the range $\{1\leq p\leq k+r-1, q=n+r-1-p, l=k+r-p\}\cup\{1\leq p\leq k+r-2, q=n+r-p-2, l=k+r-1-p\}\cup\{1\leq p\leq k+r-2, q=n+r-1-p, l=k+r-1-p\}$ and \eqref{CV2} with $p=1, q-a=n, a=r-2, l-a=k-r+2$, then we have the following description of the Hodge structure of $Y$ 
\begin{equation}
H^{n-r-k, k}_{var}(Y)\cong M^{k+r-1}/N^{k+r-1}f+JM^{k+r-2}
\end{equation}

\end{thm}
\begin{proof}
According to Proposition \ref{PCV}, we have a surjective map $H^0(\PP, K_\PP\otimes \OX(k+1))\to F^{n+r-1-k}H^{n+r-1}(\PP-\tY)$ with kernel equal to $dH^0(\PP, \Omega^{n-r-2}_\PP\otimes \OX(k))$. Using Leray spectral sequence, we have $H^0(\PP, K_\PP\otimes \OX(k+1))$ is nonzero only if $k\geq r-1$ and isomorphic to  $M^{k}$.
Now we describe $H^0(\PP, \Omega^{n-r-2}_\PP\otimes \OX(k))$. There is a natural isomorphism $\Omega^{n-r-2}_\PP\otimes \OX(k)\cong T_\PP\otimes K_\PP\otimes \OX(k)$. So we consider the exact sequence 
\begin{equation}
0\to T_v \to T_\PP\to \pi^*T_X\to 0
\end{equation} 
twisted by $K_\PP\otimes \OX(k)$.
Here $T_v=\ker \pi_*$. This induces a long exact sequence
\begin{equation}
0\to H^0(\PP, T_v\otimes K_\PP\otimes \OX(k))\to H^0(\PP, T_\PP\otimes K_\PP\otimes \OX(k))\to H^0(\PP, \pi^*T_X\otimes K_\PP\otimes \OX(k))\to H^1(\PP, T_v\otimes K_\PP\otimes \OX(k))
\end{equation}
First we claim $H^1(\PP, T_v\otimes K_\PP\otimes \OX(k))=0$. There is an exact sequence
\begin{eqnarray*}
0\to H^1(X, R^0\pi_*(T_v\otimes \OX(k-r))\otimes K_X\otimes \det E\to H^1(\PP, T_v\otimes K_\PP\otimes \OX(k))\to \\
 H^0(X, R^1\pi_*(T_v\otimes \OX(k-r))\otimes K_X\otimes \det E)
\end{eqnarray*}
We have $R^1\pi_*(T_v\otimes \OX(k-r))=0$ according to Bott's vanishing theorem on $\PP^{r-1}$. According to Lemma \ref{fiber}, the bundle $R^0\pi_*(T_v\otimes \OX(k-r))\otimes K_X\otimes \det E$ is a direct summand of $\wedge^{r-2}E\otimes S^{k-r+2}E\otimes K_X$. So $H^1(X, R^0\pi_*(T_v\otimes \OX(k-r))\otimes K_X\otimes \det E)$ also vanishes due to condition \ref{CV2}.

Now we describe $H^0(\PP, T_v\otimes K_\PP\otimes \OX(k))$. Consider the Euler sequence 
\begin{equation}
0\to\CC\to \pi^*E^\vee\otimes\OX(1)\to T_v\to 0
\end{equation}
twisted by $K_\PP\otimes \OX(k)$. Then we have an surjective map
\begin{equation}
N^{k}\to H^0(\PP, T_v\otimes K_\PP\otimes \OX(k))
\end{equation}
since $H^1(\PP, K_\PP\otimes\OX(k))=0$.
The image under the map $d\colon H^0(\PP, \Omega^{n+r-2}(k\tY))\to M^k/M^{k-1}f$ is described as follows. The paring of $E^\vee \otimes E\to \CC$ induces a pairing $H^0(X, E^\vee \otimes S^{k+1-r}(E)\otimes  \det E \otimes K_X)\otimes H^0(X, E)\to H^0(X, S^{k+1-r}(E)\otimes  \det E \otimes K_X) $. Since $f\in H^0(X, E)$, this gives a map $N_k\to M_k$ by pairing with $f$.

Next we describe $H^0(\PP, \pi^*T_X\otimes K_\PP\otimes \OX(k))$. It is zero when $k=r-1$ and isomorphic to $H^0(X, T_X\otimes S^{k-r} (E)\otimes K_X\otimes \det E))$ when $k\geq r$. Consider the exact sequence 
\begin{equation}
0\to \fp\to \fg\to \fg/\fp\to 0
\end{equation}
twisted by $S^{k-r} (E)\otimes K_X\otimes \det E$.
Under condition \ref{C1V}, we have a surjective map 
\begin{equation}
\fg \otimes H^0(X, S^{k-r} (E)\otimes K_X\otimes \det E)\to H^0(X, T_X\otimes S^{k-r} (E)\otimes K_X\otimes \det E))
\end{equation}
The image of $\fg \otimes H^0(X, S^{k-r} (E)\otimes K_X\otimes \det E)$ under $d$ to $M^k/M^{k-1}$ is described as Lie derivative of $\fg$ on $M^k$. So $F^{n+r-k-1}H^{n+r-1}(X-Y)/F^{n+r-k}H^{n+r-1}(X-Y)\cong M^{k}/N^{k}f+JM^{k-1}$. 
\end{proof}

\begin{rem}
We have a natural pairing $E\otimes S^{k-r}E\to S^{k-r+1}E$. This induces a map $M_{k-1}\to N_{k}$ and commute with the pairing with $H^0(X, E)$. So the Jacobian ideal $J$ in Definition \ref{MJN} can be replaced by $J^\prime$ generated by $L_Zf$ for $Z\in \fg$.
\end{rem}

\section{Hypersurfaces in log homogenous varieties}
\label{log}
Let $X^n$ be a smooth projective variety with simple normal crossing divisor $D$. The log tangent bundle $T_X(-\log D)$ is a subsheaf of $T_X$ defined as follows. If $z_1, \cdots, z_n$ is the local coordinate of $X$ and $D$ is the hyperplanes defined by $z_1=0,\cdots, z_r=0$, then the generating sections of $T_X(-\log D)$ are $z_1\partial_1,\cdots, z_r\partial_r, \partial z_{r+1}, \cdots, \partial z_n$. We say $X$ is log homogenous if $T_X(-\log D)$ is globally generated and log parallelizable if $T_X(-\log D)$ is trivial. Toric varieties and flag varieties are examples of log homogenous varieties. See \cite{brionlog, brionvan} for discussion of log homogenous varieties. Let $\fg$ be $H^0(X, T_X(-\log D))$ and $G$ be a corresponding Lie group making $X$ as an $G$-variety. Let $X_0=X-D$ be the open $G$-orbit. Let $L$ be a $G$-equivariant line bundle on $X$. The section $f\in H^0(X, L)$ defines a hypersurface $Y$. We say it is nondegenerate if $Y+D$ is still simple normal crossing. Let $Y_0=Y-D$. In this section we discuss similar Jacobian ring description of Hodge groups of $X_0-Y_0$. It is a straightforward result following Batyrev's work on the mixed hodge structures of affine hypersurfaces in algebraic tori \cite{batyrev}.

The same as definition \ref{maindef}. Let $R$ be the graded ring $R=\oplus_{k\geq 0} H^0(X, L^k)$. The generalized Jacobian ideal $J$ is the graded ideal generated by $f, L_{Z} f$ for $Z\in \fg$. Then $M=\oplus_{k\geq 0} H^0(X, K_X(D)\otimes L^{k+1})$ is a graded $R$-module. Let $\Omega_{X,D}=\Omega_X(\log D)$. Define $R_X$ to be the kernel of action map $\OX_X\otimes \fg\to T_X(-\log D)$ following the notation in \cite{brionvan}. We consider the following two vanishing conditions. 
\begin{equation}
\label{Clog}
H^p(X, \Omega_{X,D}^q\otimes L^{l})=0 \text{ for $p>0, q\geq 0, l\geq 1$} 
\end{equation}
\begin{equation}
\label{C1log}
H^1(X, R_X \otimes  K_X(D) \otimes L^{k})=0 \text{ for } k\geq 1
\end{equation}

\begin{thm}
\label{mainlog}
Under conditions \eqref{C1log} and \eqref{Clog} for $\{p=k-l+1, q=n-k+l-1, 1\leq l \leq k\}\cup \{p=k-l+1, q=n-k+l, 0\leq l \leq k\}\cup \{p=k-l, q=n-k+l-1, 1\leq l \leq k-1\}\cup\{p=k-l, q=n-k+l, 1\leq l \leq k-1\}$, there is an isomorphism $$(M/JM)^k\to F^{n-k}H^{n}(X_0-Y_0)/F^{n-k+1}H^{n}(X_0-Y_0).$$
\end{thm}

\begin{proof}
The proof follows directly from section 6 of \cite{batyrev}. Since $Y+D$ is simple normal crossing, the Hodge to de-Rham spectral sequence degenerate at $E_1$-page. So $F^{n-k}H^{n}(X_0-Y_0)/F^{n-k+1}H^{n}(X_0-Y_0)\cong H^{k}(X, \Omega_{X}^{n-k}(\log (Y+D)))$. There is an exact sequence (Theorem 6.2 in \cite{batyrev})
\begin{eqnarray*}
0\to \Omega_{X}^{n-k}(\log (Y+D))\to \Omega_{X,D}^{n-k}(Y)\xrightarrow{d} \Omega_{X,D}^{n-k+1}(2Y)/\Omega_{X,D}^{n-k+1}(Y)\xrightarrow{d}  \\\cdots \xrightarrow{d}  \Omega_{X,D}^{n}((k+1)Y)/\Omega_{X,D}^{n}(kY)\to 0
\end{eqnarray*}
The assumption \eqref{Clog} implies $H^{k}(X, \Omega_{X,D}^{n-k}(Y))=H^{k-1}(X, \Omega_{X,D}^{n-k}(Y))=0$, $$H^{k-l}(X, \Omega_{X, D}^{n-k+l}((l+1)Y)/\Omega_{X,D}^{n-k+l}(lY))=0, 1\leq l\leq k-1$$ $$H^{k-l-1}(X, \Omega_{X,D}^{n-k+l}((l+1)Y)/\Omega_{X,D}^{n-k+l}(lY))=0, 1\leq l\leq k-2$$ and also subjectivity of the maps
$$
H^0(X, \Omega_{X,D}^{n-1}(kY))\to H^0(X, \Omega_{X,D}^{n-1}(kY)/\Omega_{X,D}^{n-1}((k-1)Y))
$$
$$
H^0(X, \Omega_{X,D}^{n}((k+1)Y))\to H^0(X, \Omega_{X,D}^{n}((k+1)Y)/\Omega_{X,D}^{n-1}((kY)).
$$
A standard spectral sequence argument shows 
\begin{equation}
H^{k}(X, \Omega_{X}^{n-k}(\log (Y+D)))\cong {H^0(X, \Omega_{X,D}^{n-1}((k+1)Y))\over H^0(X, \Omega_{X,D}^{n}(kY))+dH^0(X, \Omega_{X,D}^{n-1}(kY))}
\end{equation}
Using isomorphism $\Omega_{X,D}^{n-1}\cong T_X(-\log D)\otimes K_X(D)$ and the same calculation in the proof of Proposition \ref{map}, we prove the theorem.
\end{proof}
\begin{rem}
The vanishing condition \eqref{Clog} with $p>q$ is proved in \cite{brionvan}. See also \cite{brionvan} or \cite{norimatsu} for the log version of Lefschetz theorem for the Hodge structures other than middle dimension case.
\end{rem}
\begin{rem}
When $X$ is log parallelizable, the condition \eqref{Clog} always holds and $M =\oplus_{k\geq 0} H^0(X, L^{k+1})$ is the maximal homogeneous ideal in $R$. See \cite{brionlog} for classification of log parallelizable varieties. In toric case, this is Theorem 6.9 in \cite{batyrev}. The proof given here is the same as the proof in \cite{batyrev}.
\end{rem}



\section{Hodge conjecture for very generic hypersurfaces}
\label{Hodgeconjecture}
This section is another application of the Jacobian ring for generalized flag varieties. We prove that Hodge conjecture holds for very generic hypersurfaces in flag varieties with the vanishing conditions in section \ref{linebundles}.  
\begin{thm}[\cite{carlson2017period}, Corollary 7.5.2]
\label{NL}
Let $X=G/P$ be a generalized flag variety with odd dimension $n=2k+1$. Let $L$ be an line bundle on $X$ satisfying vanishing conditions \eqref{C} for $p,q,l$ in the range $1\leq p\leq k, q=n-p, l=k-p+1$ and $H^0(X, K_X\otimes L^k)\neq 0$. Then for $f\in H^0(X, L)$ outside a countable union of proper subvarieties, we have for the hypersurface $Y_f$
\begin{equation}
H^{k,k}(Y_f, \QQ)=\Ima \{H^{k,k}(X, \QQ)\to H^{2k}(Y_f,\QQ)\}
\end{equation}
\end{thm}

\begin{proof}
The vanishing condition \eqref{C} for $1\leq p\leq k, q=n-p, l=k-p+1$ implies that $\alpha_k\colon H^0(X, K_X\otimes L^{k+1})\to F^k H^{2k}(Y_f)$ is surjective.
According to Theorem 7.5.1, Corollary 7.5.2 in \cite{carlson2017period}, we only need to check 
\begin{equation}
\label{multiplicationmap}
H^0(X, L)\otimes H^0(X, K_X\otimes L^k)\to H^0(X, K_X\otimes L^{k+1})
\end{equation} is surjective. Since $H^0(X, K_X\otimes L^k)$, $H^0(X, K_X\otimes L^{k+1})$ and $H^0(X, L)$ are irreducible representations of $\fg$, the multiplication map is surjective if $H^0(X, K_X\otimes L^k)$ is not zero.
\end{proof}
Since the Hodge conjecture holds for generalized flag variety, we have
\begin{thm}
Under the same assumption in Theorem \ref{NL}, then Hodge conjecture for $H^{k,k}$ holds for hypersurfaces $Y_f$ outside a countable union of proper subvarieties in the linear system $|L|$.
\end{thm}

\begin{rem}
The Hodge conjecture for very generic hypersurfaces in toric varieties with certain combinatorial property is proved in \cite{bruzzo2017hodge}. The proof reduces to the surjectivity of similar map \eqref{multiplicationmap} in toric Jacobian ring. We hope similar results hold for certain log homogenous varieties from the Jacobian ring in section \ref{log}.
\end{rem}

\bibliography{reference2.bib}
\bibliographystyle{abbrv}

\end{document}